\documentclass[12pt]{amsart}

\usepackage{amsmath,amsfonts,amssymb,amscd,amsthm}
\usepackage{mathtools}
\usepackage{latexsym}
\usepackage{exscale}
\usepackage{lscape}
\usepackage{mathrsfs}
\usepackage{esint}
\usepackage[colorlinks=true, pdfstartview=FitV, linkcolor=blue, citecolor=blue, urlcolor=blue]{hyperref}

\newtheorem{theorem}{Theorem}[section]
\newtheorem{lemma}[theorem]{\bf Lemma}

\newtheorem{prop}[theorem]{Proposition}
\newtheorem{definition}[theorem]{Definition}

\newtheorem{rem}[theorem]{Remark}

\newtheorem{remark}[theorem]{\textbf{Remark}}

\allowdisplaybreaks

\numberwithin{equation}{section}

\newcommand{\R}{\mathbb R}

\newcommand{\N}{\mathbb N}

\newcommand{\bu}{\mathbf u}

\newcommand{\bh}{\mathbf h}
\newcommand{\bg}{\mathbf g}
\newcommand{\bw}{\mathbf w}
\newcommand{\bv}{\mathbf v}

\newcommand{\br}{\mathbf r}
\newcommand{\bs}{\mathbf s}

\newcommand{\vphi}{\varphi}

\newcommand{\vf}{\mathbf f}

\DeclareMathOperator*{\essinf}{ess\,inf}
\DeclareMathOperator*{\esssup}{ess\,sup}

\newcommand{\T}{\mathcal{T}}
\newcommand{\op}{\mathrm{op}}

\DeclareMathOperator{\Div}{div}

\def\w{{\bf w}}

\def\g{{\bf g}}

\newcommand{\Lp}{L^p}

\newcommand{\pp}{{p(\cdot)}}
\newcommand{\cpp}{{p'(\cdot)}}
\newcommand{\Lpp}{L^{p(\cdot)}}
\newcommand{\Lcpp}{L^{p'(\cdot)}}
\newcommand{\Pp}{\mathcal P}

\newcommand{\calS}{\mathcal{S}}
\newcommand{\LQpp}{\mathcal{L}_Q^\pp}

\newcommand{\HQpp}{H_Q^{1,\pp}}
\newcommand{\TildeHQpp}{\Tilde{H}_Q^{1,\pp}}
\newcommand{\aver}[1]{-\hskip-0.46cm\int_{#1}}

\def\XXint#1#2#3{{\setbox0=\hbox{$#1{#2#3}{\int}$}
     \vcenter{\hbox{$#2#3$}}\kern-.5\wd0}}

\newcommand{\bk}{\backslash}

\begin{document}

\title[Poincar\'e inequalities and Neumann problems]{Poincar\'e inequalities and Neumann problems for the variable exponent setting}

\author{David Cruz-Uribe, OFS, Michael Penrod, and Scott Rodney}


\date{August 17, 2021}

\address{David Cruz-Uribe, OFS \\
Dept. of Mathematics \\
University of Alabama \\
 Tusca\-loosa, AL 35487, USA}
\email{dcruzuribe@ua.edu}

\address{Michael Penrod \\
Dept. of Mathematics \\
University of Alabama \\
 Tuscaloosa, AL 35487, USA}
\email{mjpenrod@crimson.ua.edu}

\address{Scott Rodney\\
Dept. of Mathematics, Physics and Geology \\ 
Cape Breton University \\
Sydney, NS B1Y3V3, CA} 
\email{scott\_rodney@cbu.ca}

\thanks{D.~Cruz-Uribe is supported by 
  research funds from the Dean of the College of Arts \& Sciences, the
  University of Alabama. S.~Rodney is supported by the NSERC Discovery
  Grant program.  This paper is based on the masters thesis of the
  second author, written at Sam Houston State under the direction of
  Li-An Wang.  The authors would like to thank the anonymous referee
  for their clarifying comments and an important reference.}


\begin{abstract}
We extend the results of \cite{CURE-2017}, where we proved an equivalence between weighted Poincar\'e inequalities and
  the existence of weak solutions to a family of Neumann problems related to a
  degenerate $p$-Laplacian. Here we prove a similar equivalence 
  between Poincar\'e inequalities in variable exponent  spaces and
  solutions to a degenerate $\pp$-Laplacian,  a
  non-linear elliptic equation with nonstandard growth conditions.
\end{abstract}

\maketitle

\section{Introduction}\label{section:intro}
Poincar\'e inequalities play a central role in the study of regularity
for elliptic equations.  For specific degenerate elliptic equations, an
important problem is to show the existence of  such an inequality;
however, an extensive theory has been developed by assuming their
existence.  See, for example, \cite{MR2204824,MR2574880}.
In~\cite{CURE-2017}, the first and third authors, along with E.~Rosta,
gave a characterization of the existence of a weighted Poincar\'e
inequality, adapted to the solution space of degenerate elliptic equations, in
terms of the existence and regularity of a weak solution to a Neumann
problem for a degenerate $p$-Laplacian equation.

The goal of the present paper is to extend this result to the setting
of variable exponent spaces.  Here, the relevant equations are
degenerate $\pp$-Laplacians.   The basic operator is the
$\pp$-Laplacian:  given an
exponent function $\pp$ (see Section~\ref{section:prelim} below), let
\[ \Delta_\pp u = -\Div ( |\nabla u|^{\pp-2} \nabla u ). \]
This operator arises in the calculus of variations as an example of
nonstandard growth conditions, and has been extensively studied by a
number of authors: see~\cite{MR3308513, MR2639204,
  MR2291779, MR3379920} and the extensive references they contain.  We are
interested in a degenerate version of this operator,
\[ Lu = - \Div(|\sqrt{Q}\nabla u|^{\pp-2} Q\nabla u), \]
where $Q$ is a $n\times n$, positive semi-definite, self-adjoint, measurable
matrix function.   These operators have also been studied, though
nowhere nearly as extensively:  see, for instance, \cite{MR3585054,  MR2670139,
 MR3974098}.    This paper is part of an ongoing project to 
develop  a general regularity theory for these operators. 

In order to state our main result, we first give some definitions and
notation that will be used throughout our work.  Let $\Omega \subset \R^n$ be a fixed
domain (open and connected), and let $E$ be a bounded subdomain with
$\overline{E}\subseteq \Omega$.   Given an exponent function $\pp$, we
let $L^\pp(E)$ denote the associated  variable Lebesgue space; for a precise
definition, see Section~\ref{section:prelim} below.

 Let $\calS_n$ denote the collection of
all positive semi-definite, $n\times n$ self-adjoint matrices. Let
$Q: \Omega \rightarrow \calS_n$ be a measurable, matrix-valued function
whose entries are Lebesgue measurable.   We define
$$\gamma(x) = |Q(x)|_\textrm{op} = \displaystyle\sup_{|\xi|=1}
|Q(x)\xi|$$
to be the pointwise operator norm of $Q(x)$; this function will play
an important role in our results.  We will generally assume that
$\gamma^{1/2}$ lies in the variable Lebesgue space $L^{\pp}(E)$.  More
generally, let $v$ be a weight on $\Omega$: i.e., a non-negative
function in $L^1_\text{loc}(\Omega)$. Given a function $f$ on $E$, we
define the weighted average of $f$ on $E$ by
\[ f_{E,v} = \frac{1}{v(E)} \int_E f(x) v(x) dx = \aver{E} fdv.\]
If $v=1$ we write simply $f_E$.  Again, we will generally assume that
$v \in L^{\pp}(E)$.

\begin{remark}
  In this paper we do not assume any connection between weight $v$ and the
  matrix $Q$.  However, in many situations it is common to assume that
  $v$ is the largest eigenvalue of $Q$:  that is, $v=|Q|_{\op}$.  See, for
  instance, \cite{MR3544941, cruzuribe2020bounded}.
\end{remark}

\medskip

The next two definitions are central to our main result.  

\begin{definition} \label{PC-prop} Given $\pp \in \Pp(E)$, a weight
  $v$ and a measurable, matrix-valued function $Q$, suppose $v,\,
  \gamma^{1/2} \in L^\pp(E)$.  Then the pair $(v,Q)$
  is said to have the Poincar\'e property of order $\pp$ on $E$ if
  there is a positive constant $C_0=C_0(E,\pp)$ such that for all
  $f\in C^1(\overline{E})$,
\begin{equation} \label{PC-ineq}
\|f-f_{E,v}\|_{\Lpp(v;E)}
\leq  C_0 \|\nabla f\|_{\LQpp(E)}.
\end{equation}
\end{definition}

\begin{remark}
The assumption that $ v,\, \gamma^{1/2} \in L^\pp(E)$ ensures that both
  sides of inequality~\eqref{PC-ineq} are finite.
\end{remark}

\begin{definition}\label{N-prop}
  Given $\pp \in \Pp(E)$, a weight
  $v$ and a measurable matrix-valued function $Q$, suppose $v,\,
  \gamma^{1/2} \in L^\pp(E)$.  Then the pair $(v,Q)$ is said to
  have the $\pp$-Neumann property on $E$ if the following hold:
\begin{enumerate}
\item Given any $f\in \Lpp(v;E)$, there exists a weak solution
  $(u,\g)_f \in \TildeHQpp(v;E)$ to the degenerate
  Neumann problem
\begin{equation}\label{nprob}
\begin{cases}
\Div\Big(\Big|\sqrt{Q(x)}\nabla u(x)\Big|^{p(x)-2} Q(x)\nabla u(x)\Big) 
& = |f(x)|^{p(x)-2}f(x)v(x)^{p(x)} \text{ in }E\\
{\bf n}^T \cdot Q(x) \nabla u(x) &= 0\text{ on }\partial E,
\end{cases}
\end{equation}
where ${\bf n}$ is the outward unit normal vector of
$\partial E$.

\item Any weak solution $(u,{\bf g})_f \in\TildeHQpp(v;E)$ of
  \eqref{nprob} is regular:   that is, there is a positive constant
  $C_1=C_1(\pp,v,E)$ such that 
\begin{equation}\label{hypoest}
\| u\|_{\Lpp(v;E)} \leq C_1 \|f\|_{\Lpp(v;E)}^{\frac{r_*-1}{p_*-1}},
\end{equation}
where $p_*$ and $r_*$ are defined by
 \begin{equation} \label{eqn:exponents}
   p_* = \begin{cases}
     p_+ & \text{ if } \|\bg\|_{\LQpp(E)} <1 \\
     p_- & \text{ if }\|\bg\|_{\LQpp(E)} \geq 1
   \end{cases}
   \quad \text{ and } \quad
   r_* = \begin{cases}
     p_+ & \text{ if } \|f\|_{\Lpp(v;E)} \geq 1\\
     p_- & \text{ if } \|f\|_{\Lpp(v;E)} < 1
   \end{cases}.
      \end{equation}
    \end{enumerate}
\end{definition}

\begin{remark}
  The degenerate, variable exponent Sobolev space,
$\TildeHQpp(v;E)$, will be defined in Section~\ref{section:prelim}.
Here we note that the definition will require the assumption that 
 $v,\, \gamma^{1/2} \in L^\pp(E)$.
\end{remark}

\begin{remark}
  The vector function $\bg$ should be thought of as a weak gradient of $f$; we avoid the
  notation $\nabla f$ since in the degenerate setting it is often not
  a weak derivative in the classical sense.  See the discussion after
  Definition~\ref{defsobsr}. 
\end{remark}

\begin{remark}
While the PDE in \eqref{nprob} is stated in terms of a classical Neumann problem,
we make no assumptions about the regularity of the boundary $\partial
E$ in our definition of a weak solution.  In the constant exponent
case, as  noted in~\cite[Remark~2.10]{CURE-2017}, our definition of
weak solution is equivalent to this classical formulation if we assume sufficient regularity.
\end{remark}

\medskip

Our main result shows that these two properties are equivalent under
certain minimal assumptions on the exponent function $\pp$, the
weight $v$, and the operator norm of the matrix function~$Q$.

\begin{theorem}\label{main-result}
  Let $\pp \in \Pp(E)$ with $1 < p_- \leq p_+ < \infty$.  Suppose  $v$
  is a   weight in $\Omega$ and $Q$ is a measurable
  matrix function with $v,\,\gamma^{1/2} \in \Lpp(E)$. Then the pair $(v,Q)$ has
  the Poincar\'e property of order $\pp$ on $E$ if and only if $(v,Q)$ has
  the $\pp$-Neumann property on $E$.
\end{theorem}

\medskip

This result is a generalization of the main result
in~\cite{CURE-2017}; when $\pp$ is constant Theorem~\ref{main-result}
is equivalent to it.  However, this is not immediately clear.  In the
constant exponent case, the exponent on the right-hand side of the regularity estimate corresponding
to~\eqref{hypoest} is $1$.  This is because in the constant exponent
case the PDE is homogeneous and we can normalize the equation, but
this is no longer possible in the variable exponent case.  But, in the
constant exponent case,  we have that $\frac{r_*-1}{p_*-1}=1$.  

More significantly, there is also a difference in the definition of
weighted spaces and the formulation of the
Poincar\'e inequality.  
Denote the weight that appears in 
\cite{CURE-2017} by $w$; there we assumed that $w\in
L^1(E)$ and defined a function $f$ to be in $L^p(w)$ if
\[ \int_E |f|^p w\,dx < \infty, \]
However, in the present case, if $\pp=p$ is a constant exponent, then
we have that  $f\in L^p(v;E)$ if
\[ \int_E |fv|^p\,dx = \int_E |f| v^p\,dx < \infty.  \]
Therefore, to pass between our current setting and that in
\cite{CURE-2017}, we need to define $w$ by $v=w^{1/p}$.

This leads to a substantial difference in the statement of the
Poincar\'e inequality.  In~\cite{CURE-2017} the left-hand side of the 
Poincar\'e inequality is (assuming $v=w^{1/p}$)
\[ \bigg(\int_E |f(x)-f_{E,w}|^p w\,dx\bigg)^{1/p}
  = \|f-f_{E,w}\|_{L^p(v;E)}; \]
on the other hand, in Definition~\ref{PC-prop} the left-hand side is
\[ \bigg(\int_E |f(x)-f_{E,v}|^p w\,dx\bigg)^{1/p}
  = \|f-f_{E,v}\|_{L^p(v;E)}.  \]
These would appear to be different conditions, but, in fact, these two
versions of the Poincar\'e inequality are equivalent.  Moreover, we
have that if we  use the more standard,
unweighted average $f_E$ in the Poincar\'e inequality, then this
implies Definition~\ref{PC-prop}.   The converse, however, requires a
additional assumption on $v$.   Versions of
the following result are part of the folklore of PDEs; we first
encountered it as a passing remark in~\cite{FKS}.  For completeness,
we give a proof in an appendix.

\begin{prop}\label{eqPC}
Given $1<p<\infty$ and a bounded set $E$, suppose $v\in L^p(E)$ and  set
$w=v^p$.   Then,
\[  \|f-f_{E,v}\|_{L^p(v;E)}  \approx \|f-f_{E,w}\|_{L^p(v;E)}, \]
where the implicit constants depend on $E$, $p$ and $v$.
Moreover, we also have that
\[  \|f-f_{E,v}\|_{L^p(v;E)}  \lesssim \|f-f_{E}\|_{L^p(v;E)}. \]
Finally, if we assume that $v^{-1}\in L^{p'}(E)$, then
\[  \|f-f_{E}\|_{L^p(v;E)} \lesssim \|f-f_{E,v}\|_{L^p(v;E)} . \]
\end{prop}

\begin{remark}
  The hypothesis that $v^{-1}\in L^{p'}(E)$ is satisfied, for instance, if
  we assume that $v^p$ is in the Muckenhoupt class $A_p$.
\end{remark}

\medskip

The remainder of this paper is organized as follows. In Section
\ref{section:prelim} we first state the basic definitions and
properties of exponent functions and variable Lebesgue spaces needed
for our results. We then define matrix weighted variable exponent
spaces, and use these to define the degenerate Sobolev spaces where
our solutions live.  An important technical step is proving that these spaces have
the requisite properties.   We then give the precise definition of weak solution
used in Definition~\ref{N-prop}. In Sections~\ref{section:N->PC}
and~\ref{section:PC->N} we prove Theorem~\ref{main-result}, each
section dedicated to one implication.   The proof is similar in
outline to the proof in~\cite{CURE-2017}, but differs significantly in
detail as we address the problems that arise from working in
variable exponent spaces.
Finally, in
Appendix~\ref{appendix:poincare} we prove Proposition~\ref{eqPC}. 

\section{Preliminaries}\label{section:prelim}

We begin this section by reviewing the basic definitions, notation,
and properties of exponent functions and variable Lebesgue spaces. For
complete information, we refer the interested reader to \cite{DCUVLS}.

\begin{definition}
An exponent function is a Lebesgue measurable function $\pp : E
\rightarrow [1,\infty]$. Denote the collection of all exponent
functions on $E$ by $\Pp(E)$. Define the set 
$E_\infty = \{x \in E : p(x) = \infty\}$
and let
\[p_-(E) = p_- = \essinf_{x \in E}p(x), \textrm{ and\; } p_+(E) = p_+ = \esssup_{x \in E}p(x).\]
\end{definition}
\begin{definition}
Given $\pp \in \Pp(E)$ and a Lebesgue measurable function $f$, define the modular functional (or simply the modular) associated with $\pp$ by
\[ \rho_{\pp,E}(f)= \int_{E\bk E_\infty} |f(x)|^{p(x)} dx + \|f\|_{L^\infty(E_\infty)}.\]
If $f$ is unbounded on $E_\infty$ or $f(\cdot)^{\pp} \not\in L^1(E\bk E_\infty)$ then we define $\rho_{\pp,E}(f) = +\infty$. When $|E_\infty| = 0$ we let $\|f\|_{L^\infty(E_\infty)}=0$; when $|E\bk E_\infty| = 0$, then $\rho_{\pp,E}(f) = \|f\|_{L^\infty(E_\infty)}$. In situations where there is no ambiguity we will simply write $\rho_\pp(f)$ or $\rho(f)$.
\end{definition}
\begin{definition}
  Let $\pp \in \Pp(E)$ and let $v$ be a weight on $E$.
\begin{enumerate}
    \item We define the variable Lebesgue space $\Lpp(E)$ to be the
      collection of all Lebesgue measurable functions $f : E
      \rightarrow \R$ satisfying
    \[\|f\|_{L^\pp(E)} = \inf\bigg\{\mu>0
      :\rho\left(\frac{f}{\mu}\right)\leq 1\bigg\} < \infty.\]
    
    \item We define the weighted variable Lebesgue space $\Lpp(v;E)$
      to be the collection of all Lebesgue measurable functions
      satisfying
    \[ \|f\|_{L^\pp(v;E)}= \|fv\|_{L^\pp(E)} < \infty.\]
\end{enumerate}
\end{definition}
\begin{theorem}{\cite{DCUVLS}}\label{Lpp-BSR}
  Let $\pp \in \Pp(E)$. Then $\Lpp(E)$ is a Banach space. The space
  $\Lpp(E)$ is separable if and only if $p_+ < \infty$, and 
  $\Lpp(E)$ is reflexive if and only if $1 < p_- \leq p_+ < \infty$.
\end{theorem}

The previous theorem can be extended to weighted variable Lebesgue
spaces. This will be useful when proving facts variable exponent
spaces of  vector-valued functions.   The following result was proved
in~\cite{diening-harjulehto-hasto-ruzicka2010}.  The setting there is
slightly different as they considered the spaces $L^\pp(\mu)$ where
$\mu$ is a measure.  However, if we let $d\mu = v^\pp dx$, then their
results immediately transfer into our setting, since with our
hypothesis $\mu$ is a $\sigma$-finite measure when $p_+<\infty$ (which is needed to
prove separability).

\begin{theorem}\label{weightedLpp-BSR}
  Let $\pp \in \Pp(E)$ and suppose $v\in L^\pp(E)$. Then:
  \begin{enumerate}
  \item $\Lpp(v;E)$ is a Banach space.
  \item  $\Lpp(v;E)$ is separable if $p_+ < \infty$.
  \item  $\Lpp(v;E)$ is reflexive if $1 < p_- \leq p_+ < \infty$.
  \end{enumerate}
\end{theorem}
%
% \begin{proof}
% Let $E_0\subseteq E$ be the support of the weight $v$. Then $\Lpp(v;E) = \Lpp(v;E_0)$. Define the mapping $I:\Lpp(v;E_0) \rightarrow \Lpp(E_0)$ by $I(f) = fv$. This mapping is linear and is invertible, with inverse $\ds I^{-1}(g) = \frac{g}{v}$. Note that since $E_0$ is the support of $v$, $\frac{g}{v}$ is defined. Moreover, $\|f\|_{\Lpp(v;E_0)} = \|fv\|_{\Lpp(E_0)}$, and so the map $I$ is an  isometry. Thus $I$ is an isometric isomorphism from $\Lpp(v;E_0)$ to $\Lpp(E_0)$. Since $\Lpp(v;E) = \Lpp(v;E_0)$, we also have that $\Lpp(v;E)$ and $\Lpp(E_0)$ are isometrically isomorphic. Therefore, $\Lpp(v;E)$ is separable, reflexive and a Banach space whenever $\Lpp(E_0)$ is.
% \end{proof}

A useful result about variable Lebesgue spaces is the extension of
H\"older's inequality to the variable exponent norm.

\begin{theorem}{(\cite[Theorem 2.26]{DCUVLS}, H\"older's inequality)}\label{Holder}
  Given $\pp, \cpp \in \Pp(E)$ with $ \frac{1}{p(x)} + \frac{1}{p'(x)} =1$ for a.e. $x\in E$, if $f \in \Lpp(E)$ and $g \in \Lcpp(E)$, then $fg \in L^1(E)$ and 
  \[\int_E |f(x) g(x)|dx \leq K_\pp \|f\|_{L^\pp(E)} \|g\|_{L^\cpp(E)},\]
  where $K_\pp\leq 4$ is a constant depending only on $\pp$.
\end{theorem}

The next two results are technical lemmas that we will need in the
proof of our main result.

\begin{prop}{\cite[Corollary 2.23]{DCUVLS}}\label{mod-norm:ineq}
    Given $\pp \in \Pp(E)$, suppose $|E_\infty| = 0$. If $\|f\|_{\pp} \geq 1$, then 
\[ \|f\|_{\pp}^{p_-} \leq \rho(f) \leq \|f\|_\pp^{p_+}.\]
If $0 \leq \|f\|_\pp < 1$, then 
\[\|f\|_{\pp}^{p_+} \leq \rho(f) \leq \|f\|_{\pp}^{p_-}.\]
\end{prop}

\begin{prop}{(\cite[Proposition~2.21]{DCUVLS})}\label{normalized-mod}
  Given $\pp \in \Pp(E)$, for all nontrivial $f \in \Lpp(E)$,
  $\rho(f/\|f\|_\pp) = 1$ if and only if $p_+(E/E_\infty) < \infty$.
\end{prop}

The next result generalizes the trivial identity $\|f^{p-1}\|_{p'} =
\|f\|_p^{p-1}$ to the setting of variable Lebesgue spaces.

\begin{theorem}\label{norm-ineq}  Let $E \subseteq \R^n$ and $\pp \in \Pp(E)$ with $1 < p_- \leq p_+ < \infty$, and $f$ be measurable on $E$. Then, $\||f|^{\pp-1}\|_{\Lcpp(E)} $ is finite if and only if $\|f\|_{\Lpp(E)}$ is finite. In particular,
    \begin{align}
\|f\|_{\Lpp(E)}^{l_* -1} \leq \||f|^{\pp-1}\|_{\Lcpp(E)} \leq \|f\|_{\Lpp(E)}^{b_*-1}
    \label{norm:inq2}
    \end{align}
where $l_*$ and $b*$ are given by
\[ l_* = \begin{cases} p_+ & \text{ if } \|f\|_{\Lpp(E)} < 1\\
				p_- & \text{ if } \|f\|_{\Lpp(E)} \geq 1\end{cases} \hskip2cm b_* = \begin{cases} p_- & \text{ if } \| f\|_{\Lpp(E)} < 1\\
																			p_+ & \text{ if } \|f\|_{\Lpp(E)} \geq 1.\end{cases}.\]
\end{theorem}
\begin{proof} 
Let $\mu_{p'} = \| |f|^{\pp-1} \|_{\Lcpp(E)}$ and assume $\mu_{p'} < \infty$. Then $\mu_{p'}^{1/(l_*-1)} \geq \mu_{p'}^{1/(p(x)-1)}$ for almost every $x$, and so
\begin{multline*}
  \int_{E} \left( \frac{|f(x)|}{\mu_{p'}^{1/(l_*-1)}} \right)^{p(x)} dx
   \leq  \int_E \left( \frac{|f(x)|}{\mu_{p'}^{1/(p(x)-1)}}\right)^{p(x)} dx\\
 	  \leq \int_E \left( \frac{|f(x)|^{p(x)-1} }{\mu_{p'}}\right)^{p'(x)} dx
	  = \rho_{\cpp, E}\left( \frac{|f|^{\pp-1}}{\mu_{p'}}\right).
        \end{multline*}
 Since $p_- >1$, we have $\esssup p'(x) < \infty$. Thus, by Proposition \ref{normalized-mod}, the modular above equals 1. Hence, by definition of the $\Lpp(E)$ norm, $\|f\|_{\Lpp(E)} \leq \mu_{p'}^{1/(l_*-1)}$, or equivalently,
\[ \|f\|_{\Lpp(E)}^{l_*-1} \leq \||f|^{\pp-1}\|_{\Lcpp(E)} < \infty.\]

Now let $\mu_p = \|f\|_{\Lpp(E)}$, and assume $\mu_p <\infty$. Then
the proof is essentially the same: 
$\mu_p^{b_* -1} \geq \mu_p^{1/(p'(x)-1)}$ a.e., and so
\begin{multline*}
  \int_E \left( \frac{|f(x)|^{p(x)-1} }{\mu_p^{b_*-1}}\right)^{p'(x)} dx
   \leq \int_E \left( \frac{|f(x)|^{p(x)-1}}{\mu_{p}^{1/(p'(x)-1)}}\right)^{p'(x)} dx \\
	= \int_E \left( \frac{|f(x)|}{\mu_p}\right)^{p(x)} dx 
	 = \rho_{\pp,E} \left( \frac{f}{\mu}\right).
       \end{multline*}
Since $p_+ < \infty$, by Proposition \ref{normalized-mod} the above modular equals 1. Hence, 
\[\||f|^{\pp-1} \|_{\Lcpp(E)} \leq \mu_p^{b_*-1} =  \|f\|_{\Lpp(E)}^{b_*-1}<\infty.\]
\end{proof}

\begin{remark}
  The definitions of the exponents $l_*$ and $b_*$ clearly depend on
  the given function. It will be clear from context what function
  these exponents are dependent on, so we will not express this
  explicitly in our proofs.  
\end{remark}

%\begin{theorem}\label{norm-ineq}
 % Given $\pp \in \Pp(E)$, $1 < p_- \leq p_+ < \infty$, suppose $f$ is a
  %measurable function on $E$.  Then $\|f\|_{\Lp(E)}$ is finite if and only
 % if $\||f|^{\pp-1}\|_{L^\cpp(E)}$ is finite.  Moreover,
%
    %\begin{enumerate}
%     \item  $0< \||f|^{\pp-1}\|_{L^\cpp(E)} < 1$ if and only if $0 <\|f\|_{\Lp(E)}
%       < 1$, and 
%
%    \begin{equation*}
%        \|f\|_{\Lp(E)}^{p_+ -1} \leq \||f|^{\pp-1}\|_{L^\cpp(E)} \leq \|f\|_{\Lp(E)}^{p_- -1} 
%    \end{equation*}
%
%    \item $1 \leq \||f|^{\pp-1}\|_{L^\cpp(E)} < \infty$ if and only if $ 1
%      \leq \|f\|_{\Lp(E)} < \infty$, and
      %
%    \begin{equation*}
%\|f\|_{\Lp(E)}^{p_- -1} \leq \||f|^{\pp-1}\|_{L^\cpp(E)} \leq \|f\|_{\Lp(E)}^{p_+ -1}
%\end{equation*}
%
%\end{enumerate}
%
%\end{theorem}

\medskip

We now define the matrix-weighted, vector-valued Lebesgue space $\LQpp(E)$.

\begin{definition}
  Given a measurable matrix function $Q: E \rightarrow \calS_n$ and
  $\pp \in \Pp(E)$, define the matrix-weighted variable Lebesgue space
  $\LQpp(E)$ to be the collection of all measurable, vector-valued
  functions $\vf : E \rightarrow \R^n$ satisfying
\[ \| \vf\|_{\LQpp(E)} = \| |\sqrt{Q}\vf|\|_{L^\pp(E)} < \infty.\]
\end{definition}

To construct the $Q$-weighted Sobolev spaces, and to prove existence
results for the PDE \eqref{nprob}, we show that $\LQpp(E)$ is a
separable, reflexive Banach space.
%whenever $\gamma^{1/2} \in \Lpp(v;E)$.

\begin{theorem}\label{LQsepcomplete} Let $Q:E\rightarrow \calS_n$ be a
  positive semi-definite, self-adjoint, measurable, matrix-valued
  function on $E$
 such that $ \gamma^{1/2} \in \Lpp(E).$
  Then
  $\LQpp(E)$ is a Banach space.  Moreover, it is separable if
  $p_+<\infty$ and  and reflexive if $1<p_-\leq p_+<\infty$.
\end{theorem}

The proof of Theorem \ref{LQsepcomplete} requires some basic facts from
linear algebra, as well as some results about matrix functions.  If
$x = (x_1, \ldots, x_n) \in \R^n$ and $1 \leq r \leq \infty$, we
recall the $\ell^r$ norms on $\R^n$:
\[|x|_r = \left( \sum_{j = 1}^n |x_j|^r \right)^{1/r} \textrm{ and }
  |x|_{\infty} = \sup_{1 \leq j \leq n} |x_j|.\]
When $r = 2$, $|x|_r$
is the Euclidean norm and we denote it by $|\cdot|_2 =
|\cdot|$. Recall that in finite dimensions, all norms are
equivalent. In particular, we have that for all
$x\in\R^n$,
\begin{equation}\label{fin-norm:equiv}
  |x|_2 \leq |x|_1 \leq \sqrt{n} |x|_2, \quad
  |x|_\infty  \leq |x|_2 \leq \sqrt{n} |x|_\infty, \quad
  |x|_\infty  \leq |x|_1 \leq n |x|_\infty.
\end{equation}
% \begin{lemma}\label{fin-norm:equiv}
%Let $x \in \R^n$. Then the following equivalences hold:
%\begin{align*}
   % &\bullet ~|x|_2 \leq |x|_1 \leq \sqrt{n} |x|_2, \\
    %&\bullet ~|x|_\infty  \leq |x|_2 \leq \sqrt{n} |x|_\infty \\
    %&\bullet ~|x|_\infty  \leq |x|_1 \leq n |x|_\infty 
%\end{align*}
%\end{lemma}

We say that an $n\times n$ matrix function $Q(\cdot)$ is positive
semi-definite on $E$ if for every nonzero $\xi \in \R^n$,
$ \xi^TQ(x)\xi\geq 0$ for almost every $x\in \Omega$. We say $Q$
is self-adjoint if $q_{ij}=q_{ji}$ for $1 \leq i,j \leq n$. Recall
that every finite, self-adjoint matrix is diagonalizable;  for matrix
functions this can be done measurably.
\begin{lemma}{ \cite[Lemma 2.3.5]{MR1350650}}\label{diagonalize}
  Let $Q$ be a finite, self-adjoint matrix whose entries are Lebesgue
  measurable functions on some domain $E$. Then for every $x \in E$,
  $Q(x)$ is diagonalizable, i.e. there exists a matrix $U$ whose
  entries are Lebesgue measurable functions on $E$ such that $U^T Q U$
  is a diagonal matrix and $U(x)$ is orthogonal for every $x \in E$.
\end{lemma}
Equivalently, there is a measurable diagonal matrix function $D(x)$
(whose entries are the non-negative eigenvalues of $Q(x)$) and an
orthogonal matrix function $U(x)$ such that for almost every $x \in E$
	\[ Q(x) = U^T(x) D(x) U(x) \] 
In particular, given such a matrix $Q$, we define its square root by
	\[ \sqrt{Q(x)} = U^T(x) \sqrt{D(x)} U(x),  \]
where $\sqrt{D(x)}$ takes the square root of each entry of $D(x)$ along the diagonal. 

\begin{remark}\label{evector-mble} As mentioned in {\cite[Remark
    5]{MR2574880}} and as a consequence of the proof of  Lemma \ref{diagonalize},
  the eigenvalues $\{\lambda_j(x)\}_{j=1}^n$ and eigenvectors
  $\{\bv_j(x)\}_{j=1}^n$ associated to a self-adjoint, positive
  semi-definite measurable matrix function, $Q: E \rightarrow \calS_n$
  are also measuarable functions on
  $E$.%, with measurable eigenvalues $\{\lambda_j(x)\}_{j=1}^n$, denote by $\{\bv_j(x)\}_{j=1}^n$ the corresponding orthogonal unit eigenvectors. The eigenvalues, $\lambda_j(x)$ are measurable since $Q(x)$ is measurable, while the eigenvectors, $\bv_j(x)$ may be chosen to be Lebesgue measurable.
\end{remark}

\begin{proof}[Proof of Theorem \ref{LQsepcomplete}] Since $Q(x)$ 
%is semi-definite and 
self-adjoint, by Lemma \ref{diagonalize}, $Q(x)$
  is diagonalizable. By Remark \ref{evector-mble},  let
  $\lambda_1(x), \ldots, \lambda_n(x)$ be the measurable eigenvalues of
  $Q(x)$ and let 
  $\bv_1(x), \ldots, \bv_n(x)$ be measurable eigenvectors with $|\bv_j(x)|=1$ for almost every
  $x \in E$, $1 \leq j \leq n$. Hence, $\{\bv_j(x)\}_{j=1}^n$
  forms a basis for $\R^n$ for almost every $x \in E$. Fix
  $\vf \in \LQpp(E)$; then we can write $\mathbf{f}$ as
\begin{equation}\label{vf:decomp}
\mathbf{f}(x) = \sum_{j=1}^n \Tilde{f}_j(x) \bv_j(x),
\end{equation}
where $\Tilde{f}_j=\mathbf{f}^T \bv_j$ is the $j$th component of
$\mathbf{f}$ with respect to the basis
$\{\bv_j\}_{j=1}^n$. Completeness, separability and reflexivity are a
consequence of the following equivalence of norms: for all
$\mathbf{f} \in \LQpp(E)$,
\begin{equation}
    \frac{1}{n}\sum_{j=1}^n \|\Tilde{f}_j\|_{\Lpp(\lambda_j^{1/2};E)} \leq \|\mathbf{f}\|_{\LQpp(E)} \leq \sum_{j=1}^n \|\Tilde{f}_j\|_{\Lpp(\lambda_j^{1/2};E)}. \label{ineq:1}
\end{equation}

Suppose for the moment that \eqref{ineq:1} holds. To show that
$\LQpp(E)$ is complete, let $\{\mathbf{f}_k\}_{k=1}^\infty$ be a
Cauchy sequence in $\LQpp(E)$. Let $\epsilon >0$ and choose $N \in \N$
such that for every $l,m >N$,
$\|\mathbf{f}_l - \mathbf{f}_m\|_{\LQpp(E)} < \epsilon/n$. Inequality
\eqref{ineq:1} then shows for $l,m>N$,
\[\sum_{j=1}^n \|(\mathbf{f}_l-\mathbf{f}_m)^T \bv_j \|_{\Lpp(\lambda_j^{1/2};E)} \leq n\|\mathbf{f}_l-\mathbf{f}_m\|_{\LQpp(E)} < \epsilon.\]
Thus, for  $1\leq j \leq n$, $\{\mathbf{f}_k^T \bv_j\}_{k=1}^\infty$ is Cauchy in $\Lpp(\lambda_j^{1/2};E)$. By Theorem~\ref{weightedLpp-BSR}, $\Lpp(\lambda_j^{1/2};E)$ is complete. Thus, there exists
 $\Tilde{g}_j \in \Lpp(\lambda_j^{1/2};E)$ such that, as $k\rightarrow \infty$,
 \begin{equation} \label{eqn:wtd-conv}
   \|\mathbf{f}_k^T \bv_j -\Tilde{g}_j\|_{\Lpp(\lambda^{1/2};E)} \to
   0.
 \end{equation}
Define $\mathbf{g}: E \rightarrow \R^n$ by setting $\mathbf{g}(x) =
\sum_{j=1}^n \Tilde{g}_j(x)\bv_j(x)$.  Since $\Tilde{g}_j \in
\Lpp(\lambda_j^{1/2};E)$, $1\leq j \leq n$, by~\eqref{ineq:1},
$\mathbf{g} \in \LQpp(E)$.  Furthermore, we have that
\[ 
\|\mathbf{f}_k - \mathbf{g}\|_{\LQpp(E)} \leq \sum_{j=1}^n
\|\mathbf{f}_k^T \bv_j - \Tilde{g}_j\|_{\Lpp(\lambda_j^{1/2};E)}.
\]
If we combine this with \eqref{eqn:wtd-conv}, we get that
$\mathbf{f}_k \to \mathbf{g}$ in $\LQpp(E)$. Therefore, $\LQpp(E)$ is
complete.

Similarly, \eqref{ineq:1} implies $\LQpp(E)$ is separable when 
$p_+<\infty$.  Fix $\epsilon >0$. Since $\lambda_j^{1/2} \leq \gamma^{1/2} \in L^\pp(E)$, again by Theorem~\ref{weightedLpp-BSR}, $\Lpp(\lambda_j^{1/2};E)$ is separable, and so
for each $j$ there is a countable, dense subset
$D_j \subseteq \Lpp(\lambda_j^{1/2};E)$. Thus, for each
$j = 1, \ldots, n$, there exists $d_j \in D_j$ such that
\[\|\Tilde{f}_j- d_j \|_{\Lpp(\lambda_j^{1/2};E)} < \frac{\epsilon}{n}.\]

Define $\mathbf{d} \in D_1\times\cdots\times D_n$ by setting $\mathbf{d} = \sum_{j=1}^n d_j \bv_j$. Then by~\eqref{ineq:1}, 
\[ \| \vf - \mathbf{d}\|_{\LQpp(E)}  \leq \sum_{j=1}^n \|\Tilde{f}_j - d_j\|_{\Lpp(\lambda_j^{1/2};E)}  < \epsilon.\]
Thus, $D_1\times\cdots\times D_n$ is a countable dense subset of
$\LQpp(E)$, and so $\LQpp(E)$ is separable.

Finally, \eqref{ineq:1} implies $\LQpp(E)$ is reflexive when
$1<p_-\leq p_+ <\infty$. Equation \eqref{vf:decomp}
induces the map
\[ T:\LQpp(E) \rightarrow \displaystyle\prod_{j=1}^n
  \Lpp(\lambda_j^{1/2};E), \]
defined by
$T(\vf) = (\Tilde{f}_1, \ldots, \Tilde{f}_j)$. Clearly, $T$ is
linear.  $T$ is also 
bijective because of the norm equivalence \eqref{ineq:1}. 

Finally,  $T$ is continuous:  by the norm equivalence \eqref{ineq:1}
we have that
\[\|T(\vf_k) - T(\vf)\|_{\prod_j \Lpp(\lambda_j^{1/2};E)}
  = \sum_{j=1}^n \|\Tilde{f}_{jk} -
  \Tilde{f}_j\|_{\Lpp(\lambda_j^{1/2});E}
  \leq n\|\vf_k - \vf\|_{\LQpp(E)}.\]
In the same way we have that $T^{-1}$ is continuous since
\[ \|\vf_k-\vf\|_{\LQpp(E)}\leq \sum_{j=1}^n\|\Tilde{f}_{jk} -
\Tilde{f}_j\|_{\Lpp(\lambda_j^{1/2};E)}. \]
Therefore, $\LQpp(E)$ is isomorphic to the product space
$\displaystyle\prod_{j=1}^n \Lpp(\lambda_j^{1/2};E)$.  Finite  products of
reflexive spaces are reflexive; hence,  $\LQpp(E)$ is a reflexive
Banach space.

\medskip

To complete the proof we need to prove inequality \eqref{ineq:1}. Since
$|\sqrt{Q(x)}\xi|^2 = \xi^TQ(x)\xi$ for any $\xi\in\mathbb{R}^n$ and
almost every $x \in E$,  we have that 
\begin{multline*}
  |\sqrt{Q(x)} \mathbf{f}(x)|^2
  = \sum_{j=1}^n |\Tilde{f}_j(x)\sqrt{Q(x)}\bv_j(x)|^2 \\
      =\sum_{j=1}^n |\Tilde{f}_j(x)|^2 \bv_j^T(x) Q(x) \bv_j(x)
 = \sum_{j=1}^n |\Tilde{f}_j(x)|^2 \lambda_j(x).
\end{multline*}
Hence,
\begin{equation*}|
  \sqrt{Q(x)}\mathbf{f}(x)|
  = \left(\sum_{j=1}^n |\Tilde{f}_j(x)|^2 \lambda_j(x)\right)^{1/2}
\end{equation*}
almost everywhere in $E$.

Inequality~\eqref{ineq:1} is now
straightforward to prove.  Define 
$\Tilde{\bf F}: E \rightarrow \R^n$ by 
$$\Tilde{\bf F}(x) = (|\Tilde{f}_1(x)|\lambda_1^{1/2}(x), \ldots, |\Tilde{f}_n(x)|\lambda_n^{1/2}(x)).$$
By~\eqref{vf:decomp} we have that
\begin{equation*}
  \|\mathbf{f}\|_{\LQpp(E)}
  = \| |\sqrt{Q(x)}\mathbf{f}(x)|\|_{L^\pp(E)} = \| |\Tilde{\bf F}(x)|\|_{L^\pp(E)}. 
\end{equation*}
By \eqref{fin-norm:equiv} and the triangle inequality,
\[\|| \Tilde{\bf F}(x)|\|_{L^\pp(E) }
  \leq \| |\Tilde{\bf F}(x)|_1 \|_{L^\pp(E) }
  \leq \sum_{j=1}^n \| |\Tilde{ f}_j \lambda_j^{1/2}|\|_{L^\pp(E)}
  = \sum_{j=1}^n \|\Tilde{f}_j\|_{\Lpp(\lambda_j^{1/2};E)}.\]
To show the reverse inequality, we again use~\eqref{fin-norm:equiv} and
the definition of $|\cdot|_\infty$ to get
\begin{multline*}
  \| | \Tilde{\bf F}(x)| \|_{L^\pp(E) }
  \geq \frac{1}{n} \sum_{j=1}^n \| | \Tilde{\bf F}(x)|_\infty
  \|_{L^\pp(E)} \\
  \geq \frac{1}{n} \sum_{j=1}^\infty \|
  |\Tilde{f}_j(x)\lambda_j^{1/2}(x)|\|_{L^\pp(E)}
  =\frac{1}{n} \sum_{j=1}^n \|\Tilde{f}_j\|_{\Lpp(\lambda_j^{1/2};E)}.
\end{multline*}
This completes the proof of \eqref{ineq:1}.
\end{proof}

\medskip

We now use these variable exponent spaces to define the degenerate
Sobolev spaces where solutions in Definition~\ref{N-prop} will live.
Initially, we will give them as collections of equivalence classes of
Cauchy sequences of $C^1(\overline{E})$ functions.

\begin{definition}\label{defsobsr}
  Given $\pp \in \Pp(E)$ a weight $v$ and a matrix $Q$, suppose
  $ v, \, \gamma^{1/2} \in \Lpp(E).$ Define the Sobolev space
  $\HQpp(v;E)$ to be the abstract completion of $C^1(\overline{E})$ with
  respect to the norm
  \begin{equation}\label{sobnormsr}
    \|f\|_{\HQpp(v;E)} = \|f\|_{\Lpp(v;E)} + \|\nabla f\|_{\LQpp(E)}.
  \end{equation}
\end{definition}

\begin{remark}
  With our hypotheses on $v$ and $\gamma$ this
definition makes sense, since they guarantee that for any $f\in
C^1(\overline{E})$ the right-hand side of~\eqref{sobnormsr} is finite.
\end{remark}

While this space is defined abstractly, we can give a concrete
representation of each equivalence class in it.  Since we assume
$ v, \, \gamma^{1/2} \in \Lpp(E)$, by Theorems~\ref{Lpp-BSR}
and~\ref{LQsepcomplete}, the spaces $\Lpp(v;E)$ and $\LQpp(E)$ are
complete.  Therefore, if $\{u_n\}_n$ is a sequence of $C^1(\overline{E})$
functions that is Cauchy with respect to the norm in~\eqref{sobnormsr}, we have that this
sequence is Cauchy in $\Lpp(v;E)$ and $\LQpp(E)$ and so converges to a
unique pair of functions $(u,\bg) \in \Lpp(v;E) \times \LQpp(E)$.
We stress that while the function $\bg$ plays the role of $\nabla u$, it cannot in
general be identified with a weak derivative of $u$ in the classical
sense, even in the constant exponent case.  For additional details,
see~\cite{CURE-2017}.  

\begin{theorem}\label{oversob}
  Let $\pp \in \Pp(E)$ and suppose $v,\,\gamma^{1/2}\in \Lpp(E)$. Then
  $\HQpp(v;E)$ is a Banach space.  If $p_+ < \infty$, then it
  is separable, and if $1 < p_- \leq p_+ < \infty$, it is reflexive.
\end{theorem}

\begin{proof}
  Recall that a closed subspace of a separable, reflexive Banach space
  is also a separable, reflexive Banach space. Hence, by
  Theorems~\ref{Lpp-BSR} and~\ref{LQsepcomplete}, it suffices to show
  that $\HQpp(v;E)$ is isometrically isomorphic to a closed subspace
  of $\Lpp(v;E)\times \LQpp(E)$.  Given a sequence $\{u_n\}_{n}$ of
  $C^1(\overline{E})$ functions that is Cauchy with respect to
  \eqref{sobnormsr}, denote its associated equivalence class in
  $\HQpp(E)$ by $[\{u_n\}_n]$.   Then we have that
  \[ \|[\{u_n\}_n]\|_{\HQpp(E)}
    = \lim_{n\rightarrow \infty}
    \big( \|u_n\|_{\Lpp(v;E)}+\|\nabla u_n \|_{\LQpp(E)}\big)
    = \|u\|_{\Lpp(v;E)}+\|\bg\|_{\LQpp(E)}, \]
  where the pair $(u,\bg) \in \Lpp(v;E)\times \LQpp(E)$ is the unique
  limit described above. (Note that this limit does not depend on the
  representative chosen from the equivalence class.)

  The existence of this pair lets us define a natural map
  \[ I: \HQpp(v;E) \rightarrow \Lpp(v;E)\times \LQpp(E) \]
  by $I([\{u_n\}_n]) = (u,\bg)$. 
Clearly,  $I$ is linear and an isometry by
construction.
Finally,  if $(u,\bg)$ is a
limit point of the image, then by a diagonalization argument we can
construct a sequence $\{u_n\}_n$ in $C^1(\overline{E})$ that converges
to it in the product norm.  But then the sequence is Cauchy in
$\HQpp(v;E)$ norm, and so $(u,\bg)$ is contained in the image of
$\HQpp(v;E)$.  Thus, $\HQpp(v;E)$ is isometrically isomorphic to a
closed subspace of  $\Lpp(v;E)\times \LQpp(E)$ and our proof is complete.
\end{proof}

It is well known that when considering Neumann boundary
value problems, any solution is unique only up to addition of
constants.  In other words if $u$ were a solution of the Neumann
problem $\eqref{nprob}$, then we should have that $u+c$ is also a
solution for any constant $c$.  Therefore, in defining weak solutions
we will restrict our attention to
the ``mean-zero" subspace of $\HQpp(v;E)$.  

\begin{definition} Given the space $\HQpp(v;E)$ of
  Definition~\ref{defsobsr}, we define
\[\TildeHQpp(v;E) = \left\{ (u,\bg)\in \HQpp(v;E) : \int_E u(x) v(x) dx = 0\right\}\]
\end{definition}

For our analysis we will need to prove that $\TildeHQpp(v;E)$ inherits
the properties of its parent space from Theorem~\ref{oversob}.

\begin{theorem}\label{TildeHQpp-complete}
  Given $\pp \in \Pp(E)$, suppose $v,\,\gamma^{1/2}\in \Lpp(E)$.  Then
  $\TildeHQpp(v;E)$ is a Banach space. Furthermore, $\TildeHQpp(v;E)$
  is separable if $p_+ < \infty$, and is reflexive if
  $1 < p_- \leq p_+ < \infty$.
\end{theorem}

\begin{proof}
  To
  show that $\TildeHQpp(v;E)$ is a Banach space, it will suffice to show
  that $\TildeHQpp(v;E)$ is a closed subspace of the Banach space
  $\HQpp(v;E)$. Let
  $\{(u_j, \bg_j)\}_{j=1}^\infty$ be a Cauchy sequence in
  $\TildeHQpp(v;E)$. Since $\HQpp(v;E)$ is complete, there is an
  element $(u,\bg)\in \HQpp(v;E)$ such that $u_j \to u$ in $\Lpp(v;E)$
  and $\bg_j \to \bg$ in $\LQpp(E)$. Since $u_j \in \TildeHQpp(v;E)$
  for each $j$, we have that $\int_E u_j(x) v(x) dx = 0$. Thus, by
  H\"older's inequality (Theorem~\ref{Holder}) we get
\begin{multline*}
  \left|\int_E u(x) v(x) dx \right|
  =\left| \int_E (u(x) - u_j(x))v(x)dx\right| \\
    \leq \int_E |u(x) - u_j(x)|v(x) dx 
     \leq K_\pp \|(u-u_j)v\|_{\Lpp(E)}\|1\|_{\Lcpp(E)}.  
   \end{multline*}
   
Since $E$ is bounded, $\|1\|_{\Lcpp(E)} < \infty$. This follows at
once from~\cite[Corollary~2.48]{DCUVLS}.
Since $u_j \to u$ in $\Lpp(v;E)$,  it follows that the right-hand side
converges to $0$.  Hence,
$$\int_E u(x)v(x)dx = 0$$
and so $(u,\bg) \in \TildeHQpp(v;E)$ and we have that
$\TildeHQpp(v;E)$ is a closed subspace of $\HQpp(v;E)$.

If $p_+<\infty$, then $\HQpp(v;E)$ is separable, and so every closed
subspace, in particular  $\TildeHQpp(v;E)$, is also separable.  
Finally, if $1<p_-\leq p_+<\infty$, then $\HQpp(v;E)$ is reflexive,
and since  every closed subspace of a reflexive Banach space is
reflexive, $\TildeHQpp(v;E)$ is as well.
\end{proof}

As part of the proof of Theorem \ref{main-result}, we will need to
apply the Poincar\'e inequality to any element of $\TildeHQpp(v;E)$
and not just to $C^1(\overline{E})$ functions.  To prove we can do
this, we need the following lemma.

\begin{lemma}\label{density}
  Given $\pp \in \Pp(E)$ with $p_+ < \infty$, suppose
  $v,\,\gamma^{1/2}\in \Lpp(E)$. Then the set
  $C^1(\overline{E}) \cap \TildeHQpp(v;E)$ is dense in
  $\TildeHQpp(v;E)$.
\end{lemma}
\begin{proof}
  Fix $(u,\bg) \in \TildeHQpp(v;E)$.  Since $C^1(\overline{E})$ is
  dense in $\TildeHQpp(v;E) \subseteq \HQpp(v;E)$,  there exists
  a sequence of functions $u_k \in C^1(\overline{E})$ such that
  $(u_k, \nabla u_k) \to (u, \bg)$ in norm. Let
  $y_k = u_k - (u_k)_E \in C^1(\overline{E}) \cap \TildeHQpp(v;E)$;
  then
  $\nabla y_k = \nabla u_k$, and so to prove
  $(y_k, \nabla y_k) \to (u, \bg)$ it will suffice to show
  $u_k-y_k = (u_k)_{E,v}\to 0 $ in $\Lpp(v;E)$.  Since
  $(u, \bg) \in \TildeHQpp(v;E)$, we have
  $u_{E,v}= 0$, and so by H\"older's inequality
  (Theorem \ref{Holder})
\begin{equation*}
  (u_k)_{E,v} = \frac{1}{v(E)} \int_E (u_k - u)v\,dx
  \leq  K_\pp \|u_k - u\|_{\Lpp(v;E)}\|1\|_{\Lcpp(E)}.  
\end{equation*}

Since $E$ is bounded, $\|1\|_{\Lcpp(v;E)}<\infty$ as in the previous
proof.  Thus $(u_k)_{E,v} \rightarrow 0$.  Consequently,
\[\|(u_k)_E\|_{\Lpp(v;E)} = |(u_k)_E| \|1\|_{\Lpp(v;E)}\]
converges to zero since $1\in \Lpp(v;E)$.
\end{proof}

\begin{theorem}\label{pcforH} If Definition \ref{PC-prop} holds,
  then the Poincar\'e inequality
  \[ \|u\|_{\Lpp(v;E)} \leq C_0\| \bg \|_{\LQpp(E)}\]
holds for every pair $(u,\bg) \in \TildeHQpp(v;E)$.  
\end{theorem}

\begin{proof}
  By Lemma \ref{density}, for every $(u,\bg) \in \TildeHQpp(v;E)$,
  there exists a sequence of functions
  $\{u_k\}_{k=1}^\infty \subseteq C^1(\overline{E})\cap
  \TildeHQpp(v;E)$ such that $\|u_k \|_{\Lpp(v;E)} \to \|u\|_{\Lpp(v;E)}$
  and $\|\nabla u_k\|_{\LQpp(v;E)} \to \|\bg\|_{\LQpp(v;E)}$ as
  $k \to \infty$. Since $u_k \in \TildeHQpp(v;E)$, 
  $(u_k)_E = 0$ for $k\in \mathbb{N}$. Hence, by  Definition
  \ref{PC-prop},
\begin{multline*}
  \|u \|_{\Lpp(v;E)}
 = \lim_{k \to \infty }\|u_k \|_{\Lpp(v;E)} \\
= \lim_{k \to \infty }\|u_k - (u_k)_E\|_{\Lpp(v;E)}
\leq C_0 \lim_{k \to \infty }\|\nabla u_k\|_{\LQpp(v;E)}
= C_0 \|\bg\|_{\LQpp(v;E)}.
\end{multline*}
\end{proof}

\medskip

Finally, we  define a weak solution to the degenerate
$\pp$-Laplacian from Definition~\ref{N-prop}.

\begin{definition} \label{weaksol}
  Let $E \subseteq \R^n$ be a bounded open set, $\pp \in \Pp(E)$, and
  $v,\,\gamma^{1/2}\in \Lpp(E)$. Given $f \in \Lpp(v;E)$, 
  the pair $(u,\bg)_f \in \TildeHQpp(v;E)$ is a weak solution to the
  Neumann problem \eqref{nprob} if for all test functions
  $\varphi \in C^1(\overline{E}) \cap \TildeHQpp(v;E)$,
\[\int_E | \sqrt{Q(x)}\bg(x)|^{p(x)-2} (\nabla \varphi(x))^T Q(x) \bg(x) dx = -\int_E |f(x)|^{p(x)-2}f(x) \varphi(x)(v(x))^{p(x)}dx.\]
\end{definition}

%%%%%%%%%%%%%%
%%%%%%%%%%%%%%
%%%%%%%%%%%%%%

\section{$\pp$-Neumann Implies $\pp$-Poincar\'e}\label{section:N->PC}

In this section we will give the first half of the proof of
Theorem~\ref{main-result}.  
Fix $\pp\in \Pp(E)$, $1 < p_- \leq p_+ < \infty$, let $v$ be a weight in $\Omega$ with
$v\in \Lpp(E)$, and $Q$ a measurable matrix function with
$\gamma^{1/2}\in \Lpp(E)$.   Assume that the Definition~\ref{N-prop}
holds.  We will show that the Poincar\'e inequality in
Definition~\ref{PC-prop} holds.

We begin by showing that the regularity condition~\eqref{hypoest} in
Definition~\ref{N-prop} actually implies a stronger condition.  

\begin{lemma}\label{N->PC:p*r*}
  Let $\pp$, $v$, $Q$ be as defined above.  Then there exists a 
  constant $C =C(\pp, v, E)$ such that for any $f \in \Lpp(v;E)$ and
  any corresponding weak solution $(u,\bg)_f \in \TildeHQpp(v;E)$ of
  \eqref{nprob},
\[\|\bg\|_{\LQpp(E)}^{p_* -1} \leq C\|f\|_{\Lpp(v;E)}^{r_*-1},\]
where $p_*$ and $r_*$ are defined by~\eqref{eqn:exponents}.
\end{lemma}

\begin{proof}
  Let $f\in \Lpp(v;E)$ and $(u,\bg)_f$ be a weak solution of
  \eqref{nprob} with data $f$.  By Proposition \ref{mod-norm:ineq},
  H\"older's inequality, the regularity estimate \eqref{hypoest}, and
  Theorem~\ref{norm-ineq}, and using the weak solution $(u,\bg)_f$ itself as a test function in the definition of weak solution, we have that
\begin{align*}
    \|\bg\|_{\LQpp(E)}^{p_*} & \leq \int_E |\sqrt{Q(x)}\bg(x)|^{p(x)} dx \\
    & = \int_E |\sqrt{Q(x)}\bg(x)|^{p(x)-2} \bg(x)^T Q(x) \bg(x) dx \\
    & = - \int_E |f(x)|^{p(x)-2} f(x) u(x) v(x)^{p(x)} dx \\
    & \leq \int_E |f(x)|^{p(x)-1} v(x)^{p(x)-1} |u(x)| v(x) dx \\
    & \leq K_\pp \|(fv)^{\pp-1}\|_{\Lcpp(E)} \|uv\|_{\Lpp(E)} \\
    & \leq K_\pp \|f\|_{\Lpp(v;E)}^{r_*-1} \|u\|_{\Lpp(v;E)}. \\
    & \leq K_\pp C_1\|f\|_{\Lpp(v;E)}^{r_*-1} \|f\|_{\Lpp(v;E)}^{\frac{r_*-1}{p_*-1}}.
\end{align*}
Note that in the second to last inequality, we used that fact that in
this case the exponent $b_*$ in Theorem~\ref{norm-ineq} equals $r_*$. 
Therefore, if we raise both sides to the power of $(p_*-1)/p_*$, we get
\[ \|\bg\|_{\LQpp(E)}^{p_*-1} \leq C \|f\|_{\Lpp(v;E)}^{r_*-1},\]
where $C=C(\pp,v,E)$.
\end{proof}

To prove that the Poincar\'e inequality holds, fix $f\in
C^1(\overline{E})$.  We will first consider the
special case where $f_{E,v}=0$ and  $\|f\|_{\Lpp(v;E)}=1$.   Then by
Proposition~\ref{normalized-mod}, the definition of weak solution with
$f$ as our test function, H\"older's inequality, and
Theorem~\ref{norm-ineq}, 
\begin{align*}
  \|f\|_{\Lpp(v;E)}
  &= \int_E |f(x)v(x)|^{p(x)} dx \\
    & = \int_E |f(x)|^{p(x)-2} f(x) f(x) v(x)^{p(x)} dx\\
    %& = -\int_E |\sqrt{Q(x)}\bg(x)|^{p(x)-2} \nabla f(x)^T Q(x) \bg(x) dx \\
    & \leq \int_E |\sqrt{Q(x)} \bg(x)|^{p(x)-2} |\nabla f(x)^T Q(x) \bg(x)| dx\\
    & \leq \int_E |\sqrt{Q(x)} \bg(x)|^{p(x)-1} |\sqrt{Q(x)}\nabla f(x)| dx  \\
    & \leq K_\pp \||\sqrt{Q}\bg|^{\pp-1}\|_{\cpp} \|\nabla f\|_{\LQpp(E)} \\
    & \leq K_\pp \|\bg\|_{\LQpp(E)}^{b_*-1} \|\nabla f\|_{\LQpp(E)}.
\end{align*}
By Lemma \ref{N->PC:p*r*} and our assumption that $\|f\|_{\Lpp(v;E)} =1$, we find
%
%\[ \|\bg\|_{\LQpp(E)}^{b_*-1}  \leq C^{\frac{b_*-1}{p_*-1}} \|f\|_{\Lpp(v;E)}^{\frac{(r_*-1)(b_*-1)}{p_*-1}} \leq C,\]
%
%where $C$ depends on $\pp$. Applying this, we have that
%
\[ \|f\|_{\Lpp(v;E)} \leq C \|\nabla f\|_{\LQpp(E)},\]
where $C=C(\pp,v,E)$.  This is what we wanted to prove.\\

To prove the general case, let $f_0= f-f_{v,E}$, and
$f_1 = f_0/\|f_0\|_{\Lp(v;E)}$.  Then $f_1$ has zero mean and
$\|f_1\|_{\Lp(v;E)}=1$, so by the previous case $f_1$ satisfies the
Poincar\'e inequality.  But by the homogeneity of this inequality, and
since $\|f_0\|_{\Lp(v;E)} \nabla f_1 = \nabla f_0=\nabla f$, we have
that $f$ satisfies the Poincar\'e inequality as well.  This completes
the proof.

\section{$\pp$-Poincar\'e Implies $\pp$-Neumann}\label{section:PC->N}

In this section we will give the second half of the proof of
Theorem~\ref{main-result}.  
Fix $\pp\in \Pp(E)$, $1 < p_- \leq p_+ < \infty$, let $v$ be a weight in $\Omega$ with
$v\in \Lpp(E)$, and $Q$ a measurable matrix function with
$\gamma^{1/2}\in \Lpp(E)$.   Assume that the Poincar\'e inequality in
Definition~\ref{PC-prop} holds.  We will show that
Definition~\ref{N-prop} holds by showing that a weak solution
to~\eqref{nprob} exists and that the regularity
estimate~\eqref{hypoest} is satisfied.

To show the existence of a weak solution to the  Neumann problem
\eqref{nprob}, we will apply Minty's theorem~\cite{MintysThm}.  To
state it, we introduce some notation.
Given a reflexive Banach space $\mathcal{B}$, denote its dual space by
$\mathcal{B}^*$.  Given a functional $\alpha\in \mathcal{B}^*$, write
its value at $\vphi\in \mathcal{B}$ as
$\alpha(\vphi) = \langle \alpha,\vphi\rangle$.  Thus, if
$\beta:\mathcal{B} \rightarrow\mathcal{B}^*$ and $u\in \mathcal{B}$,
then we have $\beta(u)\in \mathcal{B}^*$ and so its value at $\vphi$ is
denoted by $\beta(u)(\vphi) = \langle \beta(u),\vphi\rangle$.

\begin{theorem}{(Minty's Theorem, \cite{MintysThm})}\label{minty}
  Let $\mathcal{B}$ be a reflexive, separable Banach space and fix
  $\Gamma\in \mathcal{B}^*$.  Suppose that
  $\mathcal{T}:\mathcal{B} \rightarrow \mathcal{B}^*$ is a bounded
  operator that is:
\begin{enumerate}

\item Monotone:  $\langle
  \mathcal{T}(u)-\mathcal{T}(\vphi),u-\vphi \rangle \geq 0$ for all
  $u,\vphi \in \mathcal{B}$; 

\item Hemicontinuous:  for $z\in \mathbb{R}$, the mapping
  $z\rightarrow \langle\mathcal{T}(u+z\vphi),\vphi \rangle$ is
  continuous for all $u,\varphi \in\mathcal{B}$;

\item Almost Coercive:   there exists a constant $\lambda>0$ so
  that $\langle\mathcal{T}(u),u\rangle > \langle \Gamma,u\rangle$
  for any $u\in\mathcal{B}$ satisfying $\|u\|_\mathcal{B} >\lambda$. 

\end{enumerate}
Then the set of $u\in \mathcal{B}$ such that  $\mathcal{T}(u) = \Gamma$ is non-empty.
\end{theorem}

To apply Minty's theorem to prove the existence of a weak solution,
let $\mathcal{B} = \TildeHQpp(v;E)$.  Note that with our hypotheses,
by Theorem~\ref{TildeHQpp-complete}, $\TildeHQpp(v;E)$ is a reflexive,
separable Banach space.  
We now define the operators $\Gamma$
and $\T$ to using the right and left-hand  sides of the equation in Definition~\ref{weaksol}.

\begin{definition}\label{Gamma}
  Given $f \in \Lpp(v;E)$, define
  $\Gamma = \Gamma_f: \TildeHQpp(v;E) \rightarrow \R$ by setting
\[ \langle \Gamma,\bw \rangle = -\int_E |f(x)|^{p(x)-2} f(x) w(x) (v(x))^{p(x)} dx\]
for any $\bw=(w,\bh)\in \TildeHQpp(v;E)$. 
\end{definition}
\begin{rem} $\Gamma_f$ clearly depends on $f \in \Lpp(v;E)$. But for
  ease of notation we will simply write $\Gamma$ where $f$ is
  understood in context.
\end{rem}

\begin{definition}\label{oper-T}
Define $\T: \TildeHQpp(v;E) \rightarrow \left( \TildeHQpp(v;E)\right)^*$ by setting  
\[\langle \T(\bu), \bw\rangle = \int_E \left| \sqrt{Q(x)}\bg(x)\right| ^{p(x)-2} \bh^T(x) Q(x) \bg(x) dx \]
for $\bu = (u, \bg),~\bw = (w,\bh) \in\TildeHQpp(v;E)$,
\end{definition}

Clearly, $\bu = (u, \bg)$ is a weak solution of~\eqref{nprob} if and
only if $\langle \T(\bu), \bw\rangle = \langle \Gamma,\bw \rangle$ for
all $\bw\in \TildeHQpp(v;E)$.  Therefore, we will have shown that a
weak solution exists if we can show that the operators $\Gamma_f$ and
$\T$ satisfy the hypotheses of Minty's Theorem. 

\begin{lemma} 
Given $f \in \Lpp(v;E)$, $\Gamma= \Gamma_f$ is a bounded,  linear functional on $\TildeHQpp(v;E)$.
\end{lemma}

\begin{proof}
We first show that $\Gamma$ is linear. Let $\bu = (u, \bg), \,\bw =
(w,\bh)$ be in $\TildeHQpp(v;E)$. Then for all $\alpha,\,\beta \in \R$, 
\begin{equation*}
  \langle \Gamma, \alpha \bu + \beta \bw\rangle
 = -\int_E |f(x)|^{p(x)-2}f(x) (\alpha u(x) + \beta w(x))(v(x))^{p(x)}dx
 = \alpha \langle \Gamma, \bu \rangle + \beta \langle \Gamma, \bw
   \rangle. 
 \end{equation*}
 
 To show that $\Gamma$ is bounded, it will suffice to show that there
 exists a  constant $C=C(f, v, \pp)$ such that
\begin{equation} \label{eqn:gamma-bdd}
\left|\langle \Gamma,\bw\rangle \right| \leq C\|wv\|_{\Lp(E)},
\end{equation}
since $\|wv\|_{\Lp(E)} = \|w\|_{\Lpp(v;E)} \leq
\|\w\|_{\TildeHQpp(v;E)}$.
 By H\"older's inequality,
\begin{multline*}
  |\langle \Gamma,\bw\rangle|
   = \left|\int_E |f(x)|^{p(x)-2} f(x)w(x)v(x)^{p(x)}\,dx\right| \\
   \leq \int_E | f(x)|^{p(x)-1}v(x)^{p(x)-1}w(x)v(x)\,dx 
    \leq K_{\pp}\| |fv|^{\pp-1} \|_{L^\cpp(E)} \|wv\|_{\Lp(E)}. 
  \end{multline*}
  
Since $f \in \Lpp(v;E)$, by Theorem \ref{norm-ineq} we have that
\[ \||fv|^{\pp-1}\|_{L^\cpp(E)} \leq \|f\|_{\Lpp(v;E)}^{b_*-1}<
  \infty.\]
Therefore, if we let $C=K_\pp \||fv|^{\pp-1}\|_{L^\cpp(E)}$, we have
that~\eqref{eqn:gamma-bdd} holds.
\end{proof}

We now prove that $\T$ is bounded, monotone, hemicontinuous.

\begin{lemma}
     $\T$ is a bounded operator.
\end{lemma}

\begin{proof}
  We will prove that $\T$ is bounded by showing the operator norm of
  $T$ is uniformly bounded. The  norm of
  $\T: \TildeHQpp(v;E) \rightarrow \left( \TildeHQpp(v;E)\right)^*$ is
  given by
\[ \|\T\|_{\op} = \sup\{|\T(\bu)|_{op} : \|\bu\|_{\HQpp(v;E)}=1\}, \]
where
$|\T(\bu)|_{op} = \sup\{|\langle \T(\bu), \w\rangle |
:\|\bw\|_{\HQpp(v;E)} = 1 \}$. Thus, it will suffice to show that there
exists a constant $C$ such that for all
$\bu,\bw \in \TildeHQpp(v;E)$ with
$\|\bu\|_{\HQpp(v;E)} = \|\bw \|_{\HQpp(v;E)} = 1$,
$|\langle \T(\bu),\bw\rangle | \leq C$. 
By Theorem \ref{norm-ineq},  for any $f \in \Lpp(v;E)$,
$\||f|^{\pp-1}\|_{L^\cpp(E)} \leq \|f\|_{\Lp(E)}^{p_*-1}$.
Therefore, by  H\"older's inequality we have that
\begin{align*}
  |\langle \T(\bu),\bw\rangle|
  & = \left| \int_E |\sqrt{Q(x)}\bg(x)|^{p(x)-2} (\bh(x))^T Q(x) \bg(x) dx \right| \\
      & \leq \int_E |\sqrt{Q(x)}\bg(x)|^{p(x)-1} |\sqrt{Q(x)}\bh(x)| dx \\
    & \leq K_{\Lp(E)} \| |\sqrt{Q}\bg|^{\pp-1}\|_{L^\cpp(E)} \||\sqrt{Q}\bh\|_{\Lp(E)}\\
    & \leq K_\pp \||\sqrt{Q}\bg|\|_{\Lp(E)}^{b_*-1} \||\sqrt{Q}\bh|\|_{\Lp(E)} \\
    & = K_\pp \|\bg\|_{\LQpp(E)}^{b_*-1} \|\bh\|_{\LQpp(E)} \\
      & \leq K_\pp.
\end{align*}
 Thus, $\T$ is bounded.
\end{proof}

\begin{lemma}
   $\T$ is Monotone. 
\end{lemma}

\begin{proof}
Let $\bu=(u,\bg)$ and $\bw=(w,\bh)$ be in $\TildeHQpp(v;E)$. Then
\begin{align*}
& \langle \T(\bu)-\T(\bw), \bu - \bw\rangle\\
&\quad = \langle \T(\bu), \bu - \bw\rangle - \langle \T(\bw), \bu - \bw\rangle \\
& \quad = \int_E |\sqrt{Q}\bg|^{\pp-2} (\bg - \bh)^T Q\bg - |\sqrt{Q} \bh|^{\pp-2}(\bg - \bh)^T Q \bh\,dx \\
& \quad = \int_E (\sqrt{Q}(\bg-\bh))^T (|\sqrt{Q}\bg|^{\pp-2} \sqrt{Q}\bg) - (\sqrt{Q}(\bg-\bh))^T (|\sqrt{Q}\bh|^{\pp-2} \sqrt{Q}\bh)\,dx \\
& \quad = \int_E (\sqrt{Q}(\bg-\bh))^T\left[|\sqrt{Q}\bg|^{\pp-2} \sqrt{Q}\bg - |\sqrt{Q} \bh|^{\pp-2}\sqrt{Q}\bh \right]\,dx\\
& \quad = \int_E \langle |\sqrt{Q}\bg|^{\pp-2}\sqrt{Q}\bg - |\sqrt{Q}\bh|^{\pp-2} \sqrt{Q}\bh , \sqrt{Q}\bg - \sqrt{Q}\bh \rangle_{\R^n}\,dx,
\end{align*}
where  $\langle \cdot , \cdot \rangle_{\R^n}$ denotes the inner
product on $\R^n$. For each $x \in E$, the integrand is of the form
\[ \langle |s|^{p-2} s - |r|^{p-2}r , s-r\rangle_{\R^n},\]
where $s,r \in \R^n$ and $p>1$.  But as noted in
\cite[p.~74]{PLind} (see also~\cite[Section~4]{MR2729257}), this expression is nonnegative. Thus,
$\T$ is monotone.
\end{proof}

\begin{lemma}
   $\T$ is Hemicontinuous.
\end{lemma}

\begin{proof}
Let $z,\,y \in \R$ and let $\bu = (u,\bg)$, $\bw = (w, \bh)$ be in $\TildeHQpp(v;E)$. Define $\psi = \bg +z\bh$ and $\gamma = \bg + y \bh$. Then 
\begin{align}
    & \langle \T(\bu + z\bw) - \T(\bu u+ y \bw), \bw \rangle \nonumber\\
    & \qquad \qquad = \int_E |\sqrt{Q}\psi|^{\pp-2} \bh^T Q \psi
      - |\sqrt{Q}\gamma|^{\pp-2} \bh^T Q \gamma\,dx \nonumber\\
    & \qquad \qquad= \int_E (\sqrt{Q}\bh)^T \left[
      |\sqrt{Q}\psi|^{\pp-2}\sqrt{Q}\psi
      - |\sqrt{Q}\gamma|^{\pp-2} \sqrt{Q}\gamma\right]\, dx \nonumber\\
    & \qquad \qquad = \int_E (\sqrt{Q}\bh)^T \left[ |\br|^{\pp-2} \br
      - | \bs|^{\pp-2}\bs \right]\, dx \label{Hemi-estimate}
\end{align}
where $\br = \sqrt{Q}\psi$ and $\bs =\sqrt{Q} \gamma$. Define
$E^+ = \{x \in E : p(x) >2\}$ and $E^- = \{x \in E : p(x) \leq
2\}$. We will show that the integral \eqref{Hemi-estimate} tends to
$0$ as $z \to y$ by estimating it over $E^+$ and $E^-$
separately.

Observe that our choice of $\br, \bs$ gives
\begin{equation} \label{r-s}
\br - \bs = \sqrt{Q}(\psi - \gamma) = \sqrt{Q}(z\bh - y \bh) = (z-y)\sqrt{Q}\bh.
\end{equation}
Hence,  
\begin{equation}\label{rs-topbound}
  \||\br - \bs|\|_{L^\pp(E)}
  = |z-y|\||\sqrt{Q}\bh|\|_{\Lpp(E)} \leq |z-y|\|\bw\|_{\HQpp(v;E)}
\end{equation}

Furthermore, by an inequality from~\cite[pp. 43, 73]{PLind} (see also~\cite[Section~4]{MR2729257}), we
have that for $\br, \bs \in \R^n$ and $p>2$,
\[\left| |\br|^{p-2}\br - |\bs|^{p-2}\bs\right|
  \leq (p-1)|\br - \bs|(|\bs|^{p-2} + |\br|^{p-2}). \]
If we combine this inequality with \eqref{r-s} and apply them to \eqref{Hemi-estimate} over $E^+$, we get that 
\begin{align}
  & \left|\int_{E^+} (\sqrt{Q}\bh)^T \left[ |\br|^{p(x)-2}\br
    - |\bs|^{p(x)-2} \bs \right] \,dx \right| \nonumber \\
  & \qquad \qquad \leq \int_{E^+} | \sqrt{Q}\bh| |p(x)-1||\br
    - \bs| (|\bs|^{p(x)-2} + |\br|^{p(x)-2}) \,dx \nonumber\\
  & \qquad \leq |z-y||p_+ -1| \int_{E^+} |\sqrt{Q}\bh|^2 \left| |\bs|^{p(x)-2}
    + |\br|^{p(x)-2} \right|\, dx. \label{E+bound}
\end{align}

Since $p(x) >2$ on $E^+$, by H\"older's inequality,
Theorem~\ref{Holder}, with exponents
$\frac{\pp}{2}, \frac{\pp}{\pp-2}$ we have
\begin{multline*}
  \int_{E^+} |\sqrt{Q}\bh|^2 \left| |\bs|^{\pp-2}
    + |\br|^{\pp-2} \right|\,dx \\
  \leq K_{\pp/2} \| |\sqrt{Q} \bh|^2 \|_{L^{\pp/2}(E^+)}
  \left\| |\bs|^{\pp-2} + |\br|^{\pp-2} \right\|_{L^{\frac{\pp}{\pp-2}}(E^+)}.
\end{multline*}
By Proposition~\ref{mod-norm:ineq}, since $\||\sqrt{Q}\bh|\|_{\Lpp(E)}$
is finite, we have that 
\[ \int_{E^+} \left| |\sqrt{Q}\bh|^2 \right|^{p(x)/2} dx
  = \int_{E^+} |\sqrt{Q} \bh|^{p(x)} dx < \infty,\]
and so by the same result we have that
$\||\sqrt{Q}\bh|^2\|_{\pp/2} < \infty$.

By the triangle
inequality,
\[ \| |\bs|^{\pp-2} + |\br|^{\pp-2} \|_{L^{\frac{\pp}{\pp-2}}(E^+)}\leq
  \| |\bs|^{\pp-2}\|_{L^{\frac{\pp}{\pp-2}}(E^+)}
  + \| |\br|^{\pp-2}\|_{L^{\frac{\pp}{\pp-2}}(E^+)}. \]
Observe that
\[ |\bs|^{(\pp-2)\cdot\pp/(\pp-2)} = |\bs|^{\pp} = |\sqrt{Q}\gamma |^{\pp} = |\sqrt{Q}(\bg + y \bh)|^{\pp}.\]
Since $\bg, \bh \in \LQpp(E)$, $\| |\sqrt{Q}(\bg +
y\bh)|\|_{L^{\pp}(E^+)} < \infty$.  But then, again by Proposition~\ref{mod-norm:ineq}, 

\[\int_{E^+}\left|
    |\sqrt{Q}(\bg+y\bh)|^{p(x)-2}\right|^{p(x)/(p(x)-2)} dx
  =  \int_{E^+}|\sqrt{Q}(\bg + y \bh)|^{p(x)} dx < \infty,\]
and so $\| |\bs|^{\pp-2}\|_{L^{\frac{\pp}{\pp-2}}(E^+)}<\infty$.
Similarly, $\| |\br|^{\pp-2}\|_{L^{\frac{\pp}{\pp-2}}(E^+)}$ is finite. Therefore,
\[ \int_{E^+} |\sqrt{Q}\bh|^2 \left| |\bs|^{p(x)-2} +
    |\br|^{p(x)-2}\right|\, dx < \infty, \]
and so \eqref{E+bound}
converges to $0$ as $z \to y$.

\medskip

Now consider the domain $E^-$.   By \cite[p.~43]{PLind} (see also~\cite[Section~4]{MR2729257}) we have that  for $\br, \bs \in \R^n$ and $1 < p \leq 2$,
\[| |\bs|^{p-2} \bs - |\br|^{p-2}\br | \leq  C(p)|\bs -
  \br|^{p-1}.\]
The constant $C(p)$ varies continuously in $p$; since for $x\in
E^-$,  $1<p_-\leq p(x) \leq 2$, we must have that 
\[ C = \sup_{x \in E^-} C(p(x))< \infty.  \]
If we apply this estimate, H\"older's inequality, 
Theorem~\ref{norm-ineq}, and~\eqref{rs-topbound} to \eqref{Hemi-estimate}, we get
\begin{align*}
  & \left| \int_{E^-}
  (\sqrt{Q}\bh)^T \left[ |\br|^{p(x)-2}\br - |\bs|^{p(x)-2}\bs \right]
  \,dx \right| \\
  & \qquad \qquad \leq \int_{E^-} |\sqrt{Q}\bh|
    \left| |\br|^{p(x)-2}\br -|\bs|^{p(x)-2} \bs \right| \,dx \\
      &  \qquad \qquad \leq C \int_{E^-} |\sqrt{Q}\bh| |\bs -\br|^{p(x)-1}\, dx \\
  &  \qquad \qquad \leq C K_\pp \|\bh\|_{\LQpp(E^-)}
    \left\| |\bs -\br|^{\pp-1}\right\|_{L^\cpp(E)}  \\
  &  \qquad \qquad \leq C K_\pp \|\bh\|_{\LQpp(E^-)}
    \| |\bs-\br|\|_{\Lp(E)}^{b_*-1} \\
    & \qquad \qquad  \leq   C K_\pp \|\bh\|_{\LQpp(E^-)} (|z-y|\|\bw\|_{\HQpp(v;E)})^{b_*-1}.  \\
\end{align*}
Thus, the integral~\eqref{Hemi-estimate} converges to $0$ on $E^-$ as
$z\to y$, and so $\T$ is hemicontinuous.

\end{proof}

\begin{lemma}\label{coercive}
$\T$ is almost coercive.
\end{lemma}
\begin{rem} The proof of Lemma~\ref{coercive} is the only part of the
  proof that requires the Poincar\'e inequality \eqref{PC-ineq}.
\end{rem}
\begin{proof}
Fix $f \in \Lpp(v;E)$ and let $\Gamma=\Gamma_f$.  We need to find
$\lambda>0$ sufficiently large that for any $\bu = (u,\bg) \in
\TildeHQpp(v;E)$ such that $\|\bu\|_{\HQpp(v;E)} > \lambda$, $\langle
\T(\bu),\bu\rangle > \langle \Gamma, \bu \rangle$.   Suppose first
that  $\lambda >1+C_0$ where $C_0$ is as in Poincar\'e
inequality~\eqref{PC-ineq}.
By Lemma~\ref{density}, and since $u$ has mean zero, we have that
\begin{equation*}
  \lambda
  < \|\bu\|_{\HQpp(v;E)}
  = \|u\|_{\Lpp(v;E)} + \|\bg\|_{\LQpp(E)} \leq (C_0 +1) \|\bg\|_{\LQpp(E)}.
  \end{equation*}
  
  Hence, $\|\bg\|_{\LQpp(E)} >1$.

By Proposition \ref{mod-norm:ineq} and the Poincar\'e inequality
\eqref{PC-ineq} we have that
\begin{multline*}
  \langle \T(\bu), \bu\rangle
  = \int_E |\sqrt{Q}\bg|^{p(x)-2} \bg^T Q \bg dx \\
     = \int_E |\sqrt{Q}\bg|^{p(x)}dx 
     \geq \|\bg\|_{\LQpp(E)}^{p_-} 
    \geq \frac{1}{C_0^{p_-}} \|u\|_{\Lpp(v;E)}^{p_-} .
  \end{multline*}
Consequently, 
\begin{equation}\label{lower-bound:Tu,u}
    (C_0^{p_-} +1) \langle \T(\bu), \bu\rangle \geq
    \|\bg\|_{\LQpp(E)}^{p_-} + \|u\|_{\Lpp(v;E)}^{p_-} \geq 2^{1-p_-}\|\bu\|_{\HQpp(v;E)}^{p_-}.
\end{equation}

By H\"older's inequality,
\begin{align*}
    |\langle \Gamma, \bu\rangle| & = \left| \int_E |f(x)|^{p(x)-2} f(x) u(x) v(x)^{p(x)} dx \right|\\
    & \leq \int_E |f(x)|^{p(x)-1} v(x)^{p(x)-1} |u(x)| v(x) dx \\
    & \leq K_\pp \|(fv)^{\pp-1}\|_{L^\cpp(E) } \|u\|_{\Lpp(v;E)} \\
    & \leq K_\pp \|(fv)^{\pp-1}\|_{L^\cpp(E)} \|\bu\|_{\HQpp(v;E)}.
\end{align*}
Since $f \in \Lpp(v;E)$, $\|fv\|_{\Lp(E)} < \infty$, and so by
Theorem~\ref{norm-ineq}, 
$\||fv|^{\pp-1}\|_{L^\cpp(E)} < \infty$. Let
$C(f) = K_\pp \|(fv)^{\pp-1} \|_{L^\cpp(E)}$; then we have
\begin{multline*}
  C(f) \|\bu\|_{\HQpp(v;E)}
   = C(f)\|\bu\|_{\HQpp(v;E)}^{1-p_-} \|\bu\|_{\HQpp(v;E)}^{p_-} \\
     \leq C(f)  \|\bu\|_{\HQpp(v;E)}^{1-p_-}
     \frac{C_0^{p_-}+1}{2^{1-p_-}}
     \langle \T(\bu), \bu \rangle.
   \end{multline*}
   
   Let $\mathcal{C} = C(f)\frac{C_0^{p_-}+1}{2^{1-p_-}}$; then
   \[ |\langle \Gamma, \bu \rangle |
     \leq \mathcal{C} \|\bu\|_{\HQpp(v;E)}^{1-p_-} \langle \T(\bu), \bu\rangle. \]
Therefore,  if we further assume that $\lambda >
\mathcal{C}^{1/(p_--1)}$, then
\[ \|\bu\|_{\HQpp(v;E)} > \lambda > \mathcal{C}^{1/(p_--1)}. \]
This in turn implies $\mathcal{C}\|\bu\|_{\HQpp(v;E)}^{1-p_-}<1$.
Thus, we have that 
\[|\langle \Gamma, \bu \rangle | < \langle \T(\bu), \bu\rangle \]
and our proof is complete. 
\end{proof}

\medskip

We have now shown that the hypotheses of Minty's theorem are
satisfied, and so a weak solution exists.  To complete the proof, we
need to prove that the regularity estimate~\eqref{hypoest} holds.  This
is established in the next lemma.

\begin{lemma}\label{p*r*:reg-ineq}
There is a positive constant $C = C(\pp, E)$ such that for any $f \in \Lpp(v;E)$ and any corresponding weak solution $(u,\bg)_f \in \TildeHQpp(v;E)$,
\[\|u\|_{\Lpp(v;E)} \leq C\|f\|_{\Lpp(v;E)}^{\frac{r_*-1}{p_*-1}},\]
where $p_*$ and $r_*$ are defined by~\eqref{eqn:exponents}.
\end{lemma}

\begin{proof}
By Definition~\ref{weaksol}, Proposition \ref{mod-norm:ineq}, Theorem \ref{norm-ineq} and the Poincar\'e inequality
\eqref{PC-ineq}, we have that
\begin{align*}
  \|\bg\|_{\LQpp(E)}^{p_*}
    & \leq \int_E |\sqrt{Q}\bg|^{p(x)-2}\bg^T Q \bg dx \\
    & = -\int_E |f|^{p(x)-2}  f u v^{p(x)} dx \\
    & \leq \int_E |f|^{p(x)-1} v^{p(x)-1} |u| v dx \\
    & \leq K_\pp \||fv|^{\pp-1}\|_{L^\cpp(E) } \|uv\|_{\Lp(E)} \\
    & \leq K_\pp C_0 \||fv|^{\pp-1}\|_{L^\cpp(E)} \|\bg\|_{\LQpp(E)} \\
    & \leq K_\pp C_0 \|fv\|_{\Lp(E)}^{r_*-1} \|\bg\|_{\LQpp(E)} \\
    & = K_\pp C_0 \|f\|_{\Lpp(v;E)}^{r_*-1} \|\bg\|_{\LQpp(E)}. 
\end{align*}
If we combine this inequality with the Poincar\'e inequality, we get
\begin{multline*}
  \|u\|_{\Lpp(v;E)} \leq C_0 \|\bg\|_{\LQpp(E)} \\
  \leq C_0( K_\pp C_0)^{1/(p_*-1)}\|f\|_{\Lpp(v;E)}^{(r_*-1)/(p_*-1)}
  \leq C_0( K_\pp C_0)^{1/(p_--1)}\|f\|_{\Lpp(v;E)}^{(r_*-1)/(p_*-1)},
\end{multline*}
which is the desired inequality. 
\end{proof}

\appendix

\section{Estimates for the weighted Poincar\'e inequality}
\label{appendix:poincare}

In this section we prove Proposition~\ref{eqPC}.  As we noted above,
versions of this result appear to be known, but we have not found the
proof in the literature.
Recall that $E$ is bounded, $v\in L^p(E)$, and $w=v^p$, so $w\in
L^1(E)$.  Fix $f\in C^1(\overline{E})$; then 
\begin{multline*}
  \left| f_{E,w} - f_{E,v}\right|
= \left| \frac{1}{v(E)}\int_E (f_{E,w} - f)w^{1/p}\,dx\right| \\
\leq K_1\left(\int_E |f-f_{E,w}|^pw\,dx\right)^{1/p}
= K_1  \|f-f_{E,w}\|_{L^p(v;E)},
\end{multline*}
where $K_1= \frac{|E|^{1/p'}}{v(E)}$. But then we have that
\begin{multline*}
  \|f-f_{E,v}\|_{L^p(v;E)}
 \leq \|f-f_{E,w}\|_{L^p(v;E)}+\|f_{E,w}-f_{E,v}\|_{L^p(v;E)} \\
\leq (1+K_1w(E)^{1/p}) \|f-f_{E,w}\|_{L^p(v;E)}.
\end{multline*}

Conversely, if we switch the roles of $v$ and $w$ in the first
calculation above, we get that 

\begin{equation*}
  \left| f_{E,w} - f_{E,v}\right|
   = \frac{1}{w(E)}\int_E (f_{E,v} - f)vv^{p-1}\,dx
\leq K_2 \|f-f_{E,v}\|_{L^p(v;E)},
\end{equation*}
where $K_2 = w(E)^{-{1/p}}$.  Then we can argue as we did before to get
\begin{equation*}
 \|f-f_{E,w}\|_{L^p(v;E)}\\
  \leq (1+K_2w(E)^{1/p}) \|f-f_{E,v}\|_{L^p(v;E)}
  = 2 \|f-f_{E,v}\|_{L^p(v;E)}. 
\end{equation*}

Similarly, if we take $w=1$ in the first argument, we get that
\[ \|f-f_{E,v}\|_{L^p(v;E)}
  \leq (1+K_3) \|f-f_{E}\|_{L^p(v;E)}, \]
where $K_3= \frac{|E|}{v(E)}$.  On the other hand, to prove the
converse inequality, we have
\[ |f_{E,v}-f_E|
  = \bigg|\frac{1}{|E|} \int_E (f-f_{E,v}) v v^{-1}\,dx\bigg|
  \leq K_4 \|f-f_{E,v}\|_{L^p(v;E)}, \]
where
\[  K_4= \frac{1}{|E|} \bigg(\int_E v^{-p'}\,dx\bigg)^{1/p'}< \infty \]
by our assumption that $v^{-1}\in L^{p'}(E)$.  The argument then
continues as before.

 \bibliographystyle{plain}

\begin{thebibliography}{10}

\bibitem{MR2729257}
M.~Biegert.
\newblock A priori estimates for the difference of solutions to quasi-linear
  elliptic equations.
\newblock {\em Manuscripta Math.}, 133(3-4):273--306, 2010.

\bibitem{DCUVLS}
D.~Cruz-Uribe and A.~Fiorenza.
\newblock {\em Variable {L}ebesgue spaces}.
\newblock Applied and Numerical Harmonic Analysis. Birkh\"{a}user/Springer,
  Heidelberg, 2013.
\newblock Foundations and harmonic analysis.

\bibitem{MR3544941}
D.~Cruz-Uribe, K.~Moen, and S.~Rodney.
\newblock Matrix {$\mathcal A_p$} weights, degenerate {S}obolev spaces, and
  mappings of finite distortion.
\newblock {\em J. Geom. Anal.}, 26(4):2797--2830, 2016.

\bibitem{cruzuribe2020bounded}
D.~Cruz-Uribe and S.~Rodney.
\newblock Bounded weak solutions to elliptic {PDE} with data in {O}rlicz
  spaces.
\newblock {\em preprint}, 2020.
\newblock arXiv:2011.14491.

\bibitem{CURE-2017}
D.~Cruz-Uribe, S.~Rodney, and E.~Rosta.
\newblock Poincar\'e inequalities and neumann problems for the $p$-laplacian.
\newblock {\em Bull. Can. Math. Soc.}, 2018.

\bibitem{diening-harjulehto-hasto-ruzicka2010}
L.~Diening, P.~Harjulehto, P.~H\"ast\"o, and M.~R{\r{u}}{\v{z}}i{\v{c}}ka.
\newblock {\em Lebesgue and {S}obolev spaces with {V}ariable {E}xponents},
  volume 2017 of {\em Lecture Notes in Mathematics}.
\newblock Springer, Heidelberg, 2011.

\bibitem{MR3308513}
D.~E. Edmunds, J.~Lang, and O.~M\'{e}ndez.
\newblock {\em Differential operators on spaces of variable integrability}.
\newblock World Scientific Publishing Co. Pte. Ltd., Hackensack, NJ, 2014.

\bibitem{FKS}
E.~B. Fabes, C.~E. Kenig, and R.~P. Serapioni.
\newblock The local regularity of solutions of degenerate elliptic equations.
\newblock {\em Comm. Partial Differential Equations}, 7(1):77--116, 1982.

\bibitem{MR2639204}
P.~Harjulehto, P.~H\"{a}st\"{o}, \'{U}t~V. L\^{e}, and M.~Nuortio.
\newblock Overview of differential equations with non-standard growth.
\newblock {\em Nonlinear Anal.}, 72(12):4551--4574, 2010.

\bibitem{MR3585054}
K.~Ho and I.~Sim.
\newblock Existence results for degenerate {$p(x)$}-{L}aplace equations with
  {L}eray-{L}ions type operators.
\newblock {\em Sci. China Math.}, 60(1):133--146, 2017.

\bibitem{MR2670139}
Y.-H. Kim, L.~Wang, and C.~Zhang.
\newblock Global bifurcation for a class of degenerate elliptic equations with
  variable exponents.
\newblock {\em J. Math. Anal. Appl.}, 371(2):624--637, 2010.

\bibitem{MR3974098}
L.~Kong.
\newblock A degenerate elliptic system with variable exponents.
\newblock {\em Sci. China Math.}, 62(7):1373--1390, 2019.

\bibitem{PLind}
P.~Lindqvist.
\newblock Notes on the p-laplace equation.
\newblock {\em preprint}, 2006.
\newblock https://folk.ntnu.no/lqvist/p-laplace.pdf.

\bibitem{MR2291779}
G.~Mingione.
\newblock Regularity of minima: an invitation to the dark side of the calculus
  of variations.
\newblock {\em Appl. Math.}, 51(4):355--426, 2006.

\bibitem{MR1350650}
A.~Ron and Z.~Shen.
\newblock Frames and stable bases for shift-invariant subspaces of {$L_2(\mathbf
  R^d)$}.
\newblock {\em Canad. J. Math.}, 47(5):1051--1094, 1995.

\bibitem{MR3379920}
V.~D. R\u{a}dulescu and D.~D. Repov\v{s}.
\newblock {\em Partial differential equations with variable exponents}.
\newblock Monographs and Research Notes in Mathematics. CRC Press, Boca Raton,
  FL, 2015.
\newblock Variational methods and qualitative analysis.

\bibitem{MR2204824}
E.~T. Sawyer and R.~L. Wheeden.
\newblock H\"{o}lder continuity of weak solutions to subelliptic equations with
  rough coefficients.
\newblock {\em Mem. Amer. Math. Soc.}, 180(847):x+157, 2006.

\bibitem{MR2574880}
E.~T. Sawyer and R.~L. Wheeden.
\newblock Degenerate {S}obolev spaces and regularity of subelliptic equations.
\newblock {\em Trans. Amer. Math. Soc.}, 362(4):1869--1906, 2010.

\bibitem{MintysThm}
R.~E. Showalter.
\newblock {\em Monotone operators in {B}anach spaces and nonlinear partial
  differential equations}, volume~49 of {\em Mathematical Surveys and
  Monographs}.
\newblock American Mathematical Society, Providence, RI, 1997.

\end{thebibliography}

% \printbibliography
\end{document}